\documentclass[12pt]{amsart}
\usepackage{amsfonts}
\usepackage{amssymb}
\usepackage{amsmath}
\usepackage{amsthm}
\usepackage{cases}
\usepackage{graphicx}
\usepackage{fullpage}

\theoremstyle{plain}
\newtheorem{thm}{Theorem}[section]
\newtheorem{lem}[thm]{Lemma}
\theoremstyle{definition}
\newtheorem{defn}[thm]{Definition}
\newtheorem{ex}[thm]{Example}
\newtheorem{rem}[thm]{Remark}

\numberwithin{equation}{section}

\newcommand{\bR}{{\mathbb R}}
\renewcommand{\a}{\alpha}
\renewcommand{\b}{\beta}
\renewcommand{\c}{\gamma}
\renewcommand{\l}{\lambda}

\renewcommand{\phi}{\varphi}
\newcommand{\ol}{\overline}

\renewcommand{\(}{\left(}
\renewcommand{\)}{\right)}
\renewcommand{\[}{\left[}
\renewcommand{\]}{\right]}

\makeatletter
\newcommand\ackname{Acknowledgements}
\if@titlepage
  \newenvironment{acknowledgements}{%
      \titlepage
      \null\vfil
      \@beginparpenalty\@lowpenalty
      \begin{center}%
        \bfseries \ackname
        \@endparpenalty\@M
      \end{center}}%
     {\par\vfil\null\endtitlepage}
\else
  \newenvironment{acknowledgements}{%
      \if@twocolumn
        \section*{\abstractname}%
      \else
        \small
        \begin{center}%
          {\bfseries \ackname\vspace{-.5em}\vspace{\z@}}%
        \end{center}%
        \quotation
      \fi}
      {\if@twocolumn\else\endquotation\fi}
\fi
\makeatother

\begin{document}
\title[Perturbed Hammerstein integral equations with deviated arguments]{Nonzero solutions of perturbed Hammerstein integral equations with deviated arguments and applications}
\date{}



\author[A. Cabada]{Alberto Cabada$^*$}\thanks{$^*$Partially supported by FEDER and Ministerio de Educaci\'on y Ciencia, Spain, project MTM2010-15314.}
\address{Alberto Cabada, Departamento de An\'alise Ma\-te\-m\'a\-ti\-ca, Facultade de Matem\'aticas, 
Universidade de Santiago de Com\-pos\-te\-la, 15782 Santiago de Compostela, Spain}%
\email{alberto.cabada@usc.es}%

\author[G. Infante]{Gennaro Infante}
\address{Gennaro Infante, Dipartimento di Matematica ed Informatica, Universit\`{a} della
Calabria, 87036 Arcavacata di Rende, Cosenza, Italy}%
\email{gennaro.infante@unical.it}%

\author[F. A. F. Tojo]{F. Adri\'an F. Tojo$^\dagger$}\thanks{$^\dagger$Supported by  FPU scholarship, Ministerio de Educaci\'on, Cultura y Deporte, Spain.}
\address{F. Adri\'an F. Tojo, Departamento de An\'alise Ma\-te\-m\'a\-ti\-ca, Facultade de Matem\'aticas,
Universidade de Santiago de Com\-pos\-te\-la, 15782 Santiago de Compostela, Spain}%
\email{fernandoadrian.fernandez@usc.es}%

\begin{abstract}
We provide a theory to establish the existence of nonzero solutions of perturbed Hammerstein integral equations with deviated arguments, being our main ingredient the theory of fixed point index.
Our approach is fairly general and covers a variety of cases.
We apply our results to a periodic boundary value problem with reflections and to a thermostat problem. In the case of reflections we also discuss the optimality of some constants that occur in our theory. Some examples are presented to illustrate the theory.
\end{abstract}

\maketitle
\section{Introduction}
The existence of solutions of boundary value problems (BVPs) with deviated arguments has been investigated recently by a number of authors using the upper and lower solutions method \cite{rub-rod-lms}, monotone iterative methods \cite{Jan1, Jan2, Sza1, Sza2}\footnote{The tight relationship between the monotone iterative method and the upper and lower solutions method has been highlighted in \cite{Cab95}. Therefore, to make a difference between them is mostly a convention.}, the classic Avery-Peterson Theorem \cite{Jan3, Jan4, Jan5, Jan6} or, in the special case of reflections, the classical fixed point index~\cite{ac-gi-at-bvp}.
One motivation for studying these problems is that they often arise when dealing with real world problems, for example when modelling the stationary distribution of the temperature of a wire of length one which is bended, see the recent paper by Figueroa and Pouso~\cite{rub-rod-lms} for details. Most of the works above mentioned are devoted to the study of \emph{positive} solutions, while in this paper we focus our attention on the existence of \emph{non-trivial} solutions. 
In particular we show how the fixed point index theory can be utilized to develop a theory for the existence of multiple non-zero solutions for
a class of 
perturbed Hammerstein integral equations with deviated arguments of the form
\begin{equation*}\label{eqHam}
u(t)=\gamma(t)\alpha[u]+\int_{a}^{b} k(t,s)g(s)f(s,u(s),u(\sigma(s)))\,ds,\quad t\in[a,b],
\end{equation*}
where 
 $\alpha[u]$ is a linear functional on $C[a,b]$ given
by
\begin{equation*}
\label{eqalpha}
\alpha[u]=\int_{a}^{b} u(s)\,dA(s),
\end{equation*}
involving a Stieltjes integral with a \emph{signed} measure, that is,
$A$ has bounded variation. 

Here $\sigma$ is a continuous function such that $\sigma([a,b])\subseteq [a,b].$
We point out that when $\sigma (t)=a+b-t$ this type of perturbed Hammerstein integral equation is well-suited to treat problems with reflections. 
Differential equations with reflection of the argument have been subject to a growing interest along the years, see for example the papers \cite{Aft, And, alb-adr, alb-adr2, alb-adr3, ac-gi-at-bvp, Gup, Gup2, Ma, Ore, Ore2, Pia, Pia2, Pia3, Wie1} and references therein. We apply our theory to prove the existence of nontrivial solutions of the first order functional periodic boundary value problem
\begin{equation}\label{eqgenpro}
u'(t)  =h(t,u(t),u(-t)),\,  t\in [-T,T];\ 
u(-T)-u(T)=\a[u],
\end{equation}
which generalises the boundary conditions in \cite{alb-adr, ac-gi-at-bvp} by adding a nonlocal term. The formulation of the nonlocal boundary conditions in terms of linear functionals is fairly general and includes, as special cases, multi-point and integral conditions, namely
$$\a[u]=\sum_{j=1}^m\a_j u(\eta_j)\quad\text{or}\quad \a[u]=\int_{0}^1\phi(s)u(s)ds.$$
The study of multi-point problems has been initiated by 1908 by Picone~\cite{Picone} and continued by a number of authors. For an introduction to nonlocal problems we refer to the reviews of Whyburn~\cite{Whyburn}, Conti~\cite{Conti}, Ma~\cite{rma}, Ntouyas~\cite{sotiris} and \v{S}tikonas~\cite{Stik} and to the papers~\cite{kttmna, ktejde, jwgi-lms}.

We also prove for the BVP \eqref{eqgenpro} the optimality of some constants that occur in our theory, improving the results even for the local case, studied in \cite{ac-gi-at-bvp}.

We study as well the existence of non-trivial solutions of the BVP
\begin{equation}\label{eq31}
u''(t)+g(t)f(t,u(t),u(\sigma(t)))=0,\ t \in (0,1),
\end{equation}
\begin{equation}\label{eq31bc}
u'(0)+{\alpha}[u]=0,\; \beta u'(1) + u(\eta)=0,\; {\eta}\in [0,1].
\end{equation}

This type of problems arises when modelling the problem of a cooling or heating system controlled by a thermostat, something that has been studied in several papers, for instance~\cite{Bro, Cam, Fri}. Nonlocal heat flow problems of the type \eqref{eq31}-\eqref{eq31bc} were studied, without the presence of deviated arguments, by Infante and Webb in~\cite{gijwnodea}, who were motivated by the previous work of Guidotti and Merino~\cite{guimer}. This study continued in a series of papers, see~\cite{Fan-Ma, gi-poit, gi-caa, gijwems, Kar-Pal, pp-gi-pp-aml-06, jwpomona, jwwcna04, jw-narwa} and references therein. The case of deviating arguments has been the subject of a recent paper by Figueroa and Pouso, see~\cite{rub-rod-lms}. In Section 4 we describe with more details the physical interpretation of the BVP~\eqref{eq31}-\eqref{eq31bc}.

We stress that the existence of nontrivial solutions of perturbed Hammerstein integral equations, without the presence of deviated arguments, namely
\begin{equation}\label{eqHam2}
u(t)=\gamma(t)\hat{\alpha}[u]+\int_{a}^{b} k(t,s)f(s,u(s))\,ds,
\end{equation}
where $\hat{\alpha}[\cdot]$ is an \emph{affine} functional given by a \emph{positive} measure, have been investigated in \cite{gijwems}, also by means of fixed point index. We make use of ideas from~\cite{gijwems} paper, but our results are somewhat different and complementary in the case of undeviated arguments.

We work in the space $C[a,b]$ of continuous functions endowed with
the usual supremum norm, and use the well-known classical
fixed point index for compact maps, we refer to the review of Amann~\cite{amann} and to the book of Guo and Lakshmikantham~\cite{guolak}
for further information.

\section{On a class of perturbed Hammerstein integral equations}
We impose the following conditions on $k$, $f$, $g$, $\c$, $\a$, $\sigma$ that occur in the integral equation
\begin{equation}\label{eqhamm}
u(t)=\gamma(t)\alpha[u]+\int_{a}^{b} k(t,s)g(s)f(s,u(s),u(\sigma(s)))\,ds =:Fu(t).
\end{equation}
\begin{enumerate}
\item [$(C_{1})$] The kernel $k$ is measurable, and for every $\tau\in
[a,b]$ we have
\begin{equation*}
\lim_{t \to \tau} |k(t,s)-k(\tau,s)|=0 \;\text{ for almost every (a.\,e.) } s \in
[a,b].
\end{equation*}{}
\item [$(C_{2})$]
 There exist a subinterval $[\hat{a},\hat{b}] \subseteq [a,b]$, a measurable function
$\Phi$ with $\Phi \geq 0$ a.\,e. in $[a,b]$ and a constant $c_1=c_1(\hat{a},\hat{b}) \in (0,1]$ such that
\begin{align*}
|k(t,s)|\leq \Phi(s) \text{ for  all }  &t \in [a,b] \text{ and a.\,e. } \, s\in [a,b],\\
k(t,s) \geq c_1\,\Phi(s) \text{ for  all } &t\in [\hat{a},\hat{b}] \text{ and a.\,e. } \, s \in [a,b].
\end{align*}{}
\item[$(C_{3})$]{}
$A$ is of bounded variation and
$\mathcal{K}_A(s):=\int_{{a}}^{{b}} k(t,s)dA(t) \geq 0
\;\text{for a.\,e.}\; s\;  \in [a,b]$.
\item [ $(C_{4})$] The function $g$ satisfies that $g\,\Phi \in L^1[a,b]$,  $g(t) \geq 0$ a.\,e. $t \in [a,b]$
and $\int_{\hat{a}}^{\hat{b}} \Phi(s)g(s)\,ds >0$.{}
\item [ $(C_{5})$]
$ 0\not\equiv \gamma \in C[a,b],\; 0 \leq \alpha[\gamma] <1\;
\text{and there exists}\; c_{2} \in(0,1] \;\text{such that}\;
\gamma(t) \geq c_{2}\|\gamma\| \;\text{for all }\; t \in [\hat{a},\hat{b}]$.
\item  [ $(C_{6})$] The nonlinearity $f:[a,b]\times (-\infty,\infty) \times  (-\infty,\infty)
 \to [0,\infty)$ satisfies Carath\'{e}odory
conditions, that is, $f(\cdot,u,v)$ is measurable for each fixed
$u$ and $v$ in $\bR$, $f(t,\cdot,\cdot)$ is continuous for a.\,e. $t\in [a,b]$, and for each $r>0$, there exists $\phi_{r} \in
L^{\infty}[a,b]$ such that{}
\begin{equation*}
f(t,u,v)\le \phi_{r}(t) \;\text{ for all } \; (u,v)\in
[-r,r]\times [-r,r],\;\text{ and a.\,e. } \; t\in [a,b].
\end{equation*}{}
\item [ $(C_{7})$]  The function $\sigma:[a,b]\to [a,b]$ is continuous.{}
\end{enumerate}
We recall that a cone $K$ in a Banach space $X$  is a closed
convex set such that $\lambda \, x\in K$ for $x \in K$ and
$\lambda\geq 0$ and $K\cap (-K)=\{0\}$.
Here we work in the cone
\begin{equation*}\label{eqcone-cs}
K=\{u\in C[a,b]: \min_{t \in [\hat{a},\hat{b}]}u(t)\geq c \|u\|,\; \alpha[u]\geq 0\},
\end{equation*}
where $c=\min\{c_1,c_2\}$ and $c_1$ and $c_2$ are given in (C2) and (C5) respectively. Note that, from (C5), $K \neq \{0\}$ since $0\ne\c\in K$ and
$$K=K_0 \cap\{u\in C[a,b]: \alpha[u] \geq 0\}, \;\text{where} \; K_{0}=\{u\in C[a,b]: \min_{t \in [\hat{a},\hat{b}]}u(t)\geq c \|u\|\}.$$

The cone $K_0$ has been essentially introduced by Infante and Webb in \cite{gijwjiea}
and later used in~\cite{ac-gi-at-bvp, dfgior1, dfgior2, Fan-Ma, giems, gi-pp1, gi-pp, gijwjmaa, gijwnodea, gijwems, nietopim}. $K_0$ is similar to
a type of cone of \emph{non-negative} functions
 first used by Krasnosel'ski\u\i{}, see e.g. \cite{krzab}, and D.~Guo, see e.g. \cite{guolak}.
Note that functions in $K_0$ are positive on the
subset $[\hat{a},\hat{b}]$ but are allowed to change sign in $[a,b]$. The cone $K$ is a modification of a cone of positive functions introduced in \cite{jwgi-nodea-08}, that allows the use of signed measures.

We require some knowledge of the classical fixed point index for
compact maps, see for example \cite{amann} or \cite{guolak} for
further information.
If $\Omega$ is a bounded open subset of $K$ (in the relative
topology) we denote by $\overline{\Omega}$ and $\partial \Omega$
the closure and the boundary relative to $K$. When $D$ is an open
bounded subset of $X$ we write $D_K=D \cap K$, an open subset of
$K$.

 Next Lemma summarises some classical results regarding the fixed point index (cf. \cite{guolak}).
\begin{lem} \label{lemind}
Let $D$ be an open bounded set with $0\in D_{K}$ and
$\overline{D}_{K}\ne K$. Assume that $F:\overline{D}_{K}\to K$ is
a compact map such that $x\neq Fx$ for all $x\in \partial D_{K}$. Then
the fixed point index $i_{K}(F, D_{K})$ has the following properties.
\begin{itemize}
\item[(1)] If there
exists $e\in K\setminus \{0\}$ such that $x\neq Fx+\lambda e$ for
all $x\in \partial D_K$ and all $\lambda
>0$, then $i_{K}(F, D_{K})=0$.
\item[(2)] If 
$\mu x \neq Fx$
for all $x\in
\partial D_K$ and for every $\mu \geq 1$, then $i_{K}(F, D_{K})=1$.
\item[(3)] If $i_K(F,D_K)\ne0$, then $F$ has a fixed point in $D_K$.
\item[(4)] Let $D^{1}$ be open in $X$ with
$\overline{D^{1}}\subset D_K$. If $i_{K}(F, D_{K})=1$ and
$i_{K}(F, D_{K}^{1})=0$, then $F$ has a fixed point in
$D_{K}\setminus \overline{D_{K}^{1}}$. The same result holds if
$i_{K}(F, D_{K})=0$ and $i_{K}(F, D_{K}^{1})=1$.

\end{itemize}
\end{lem}

\begin{defn}
Let us define the following sets for every $\rho>0$: $$K_{\rho}=\{u\in K: \|u\|<\rho\},\ 
V_\rho=\{u \in K: \displaystyle{\min_{t\in [\hat{a},\hat{b}]}}u(t)<\rho \}.$$ The set $V_\rho$ was introduced in \cite{gijwems}
and is equal to the set
called $\Omega_{\rho /c}$ in \cite{gijwjmaa}. The
notation $V_\rho$ makes it clear that
choosing $c$ as large as possible yields a weaker condition to be
satisfied by $f$ in the forthcoming Lemma \ref{idx0b1}.
A key feature 
of these sets is that they can be
nested, that is
$$
K_{\rho}\subset V_{\rho}\subset K_{\rho/c}.
$$
\end{defn}
\begin{thm}\label{thmk}
Assume that  hypotheses $(C_{1})$-$(C_{7})$ hold. Then, for every $r$, $F$ maps
$K_r$ into $K$ and is compact. Moreover $F:K\to K$ and is compact. 
\end{thm}
\begin{proof}
For $u\in
\overline{K}_{r}\text{ and }t \in [a,b]$ we have,
\begin{align*}
|Fu(t)|& \leq  |\gamma(t)|\alpha[u]+\int_{a}^{b}
|k(t,s)|g(s)f(s,u(s),u(\sigma(s)))\,ds\\
&\leq   \alpha[u]\|\gamma \|+ \int_{a}^{b} \Phi(s)g(s)f(s,u(s),u(\sigma(s)))\,ds.
\end{align*}
Taking the supremum on $t\in[a,b]$ we get
$$\|Fu\|\le\alpha[u]\|\gamma \|+ \int_{a}^{b} \Phi(s)g(s)f(s,u(s),u(\sigma(s)))\,ds$$
and, combining this fact with (C2) and (C5),
$$
 \min_{t\in [\hat{a},\hat{b}]}Fu(t) \geq c_2 \alpha[u]\|\gamma \|+c_1
\int_{a}^{b} \Phi(s)g(s)f(s,u(s),u(\sigma(s)))\,ds\geq c\|Fu\|.
$$
Furthermore, by (C3) and (C5),
$$
\alpha[Fu]=\alpha[\gamma]\alpha[u]+
\int_{a}^{b} \mathcal{K}_{A}(s)g(s)f(s,u(s),u(\sigma(s)))\,ds\geq 0.
$$

Therefore we have that $Fu\in K$ for every $u\in
\overline{K}_{r}$.

The compactness of $F$ follows from the fact that the perturbation $\gamma(t)\alpha[u]$ is compact
(since it maps a bounded set into a bounded subset of a one
dimensional space) and the fact that the Hammerstein integral operator that occurs in \eqref{eqhamm} is compact (this a consequence 
of Proposition 3.1 of Chapter 5
of~\cite{martin}). 
\end{proof}
 In the sequel, we give a condition that ensures that, for a suitable $\rho>0$, the index is 1 on $K_{\rho}$.
\begin{lem}
\label{ind1b} Assume that 
\begin{enumerate}
\item[$(\mathrm{I}_{\protect\rho }^{1})$] \label{EqB} there exists $\rho> 0$ such that
$$
 f^{-\rho,\rho}\cdot\sup_{t\in [a,b]} \left\{ \frac{|\gamma(t)|}{1-\alpha[\gamma]}\int_{a}^{b}
\mathcal{K}_{A}(s)g(s)\,ds+\int_{a}^{b}|k(t,s)|g(s)\,ds \right\}
 <1
$$
where
$$
  f^{{-\rho},{\rho}}:=\sup \left\{\frac{f(t,u,v)}{\rho }:\;(t,u,v)\in
[ a,b]\times [ -\rho,\rho ]\times [-\rho,\rho ]\right\}.$$ \end{enumerate}{}
Then the fixed point index, $i_{K}(F,K_{\rho})$, is equal to 1.
\end{lem}
\begin{proof}
We show that $\mu u \neq Fu$ for every $u \in \partial K_{\rho }$
and for every $\mu \geq 1$. 
In fact, if this does not happen, there exist $\mu \geq 1$ and $u\in
\partial K_{\rho }$ such that $\mu u=Fu$,
that is
$$\mu u(t)= \gamma(t)\alpha[u]+\int_{a}^{b}
k(t,s)g(s)f(s,u(s),u(\sigma(s)))\,ds,$$
furthermore, applying $\a$ to both sides of the equation,
$$\mu \alpha[ u]= \alpha[\gamma]\alpha[u]+\int_{a}^{b}
\mathcal{K}_{A}(s)g(s)f(s,u(s),u(\sigma(s)))\,ds,$$
thus,  from $(C_5)$, $\mu-\a[\c]\geq 1-\a[\c]>0$, and we deduce that
$$\alpha[u]=\frac{1}{\mu-\alpha[\gamma]}\int_{a}^{b}
\mathcal{K}_{A}(s)g(s)f(s,u(s),u(\sigma(s)))\,ds$$
and we get, substituting,
$$\mu u(t)= \frac{\gamma(t)}{\mu-\alpha[\gamma]}\int_{a}^{b}
\mathcal{K}_{A}(s)g(s)f(s,u(s),u(\sigma(s)))\,ds+\int_{a}^{b}
k(t,s)g(s)f(s,u(s),u(\sigma(s)))\,ds.$$

Taking the absolute value, and then the supremum for $t\in [a,b]$, gives
\begin{align*}
\mu \rho&\leq 
\sup_{t\in [a,b]} \left\{ \frac{|\gamma(t)|}{1-\alpha[\gamma]}\int_{a}^{b}
\mathcal{K}_{A}(s)g(s)f(s,u(s),u(\sigma(s)))\,ds+\int_{a}^{b}|k(t,s)|g(s)f(s,u(s),u(\sigma(s)))\,ds \right\}\\
 &\leq\rho f^{-\rho,\rho}\cdot\sup_{t\in [a,b]} \left\{ \frac{|\gamma(t)|}{1-\alpha[\gamma]}\int_{a}^{b}
\mathcal{K}_{A}(s)g(s)\,ds+\int_{a}^{b}|k(t,s)|g(s)\,ds \right\}
 <\rho.
\end{align*}

This
contradicts the fact that $\mu \geq 1$ and proves the result.
\end{proof}
\begin{rem}
We point out, in similar way as in \cite{jwgi-nodea-08}, that a stronger (but easier to check) condition than $(\mathrm{I}_{\protect\rho }^{1})$ is given by the following.
\begin{equation}
  \label{eqmest}
 f^{{-\rho},{\rho}} \; \left(\dfrac{\|\gamma\|}{1-\alpha[\gamma]}
\int_{a}^{b} \mathcal{K}_A(s)g(s)\,ds
+\dfrac{1}{m}\right)<1,
\end{equation}
where
\begin{equation}\label{mdef}\frac{1}{m}:=\sup_{t\in [a,b]}\int_{a}^{b}|k(t,s)|g(s)\,ds.\end{equation}
\end{rem}

 Let's see now a condition that warrants that the index is equal to zero on $V_\rho$ for some appropriate $\rho>0$.
\begin{lem}
\label{idx0b1} Assume that

\begin{enumerate}
\item[$(\mathrm{I}_{\protect\rho }^{0})$] there exists $\rho >0$ such that
$$
f_{\rho ,{\rho /c}}\cdot\inf_{t\in [\hat{a},\hat{b}]} \left\{ \frac{\gamma(t)}{1-\alpha[\gamma]}\int_{\hat{a}}^{\hat{b}}
\mathcal{K}_{A}(s)g(s)\,ds+\int_{\hat{a}}^{\hat{b}}k(t,s)g(s)\,ds \right\}
 >1,
$$
where
$$
f_{\rho ,{\rho /c}}: =\inf \left\{\frac{f(t,u,v)}{\rho }%
:\;(t,u,v)\in [\hat{a},\hat{b}]\times [\rho ,\rho /c]\times
[\theta,\rho /c]\right\},$$
\end{enumerate}
and
$$\theta:=\begin{cases} \rho, & \text{if}\quad \sigma([\hat a,\hat b])\subseteq [\hat a,\hat b], \\ -\rho/c, & \text{otherwise}.\end{cases}$$
Then $i_{K}(F,V_{\rho})=0$.
\end{lem}
\begin{proof}
Since $0\not\equiv\c\in K$ we can choose $e=\c$ in Lemma \ref{lemind}, so we now prove that
\begin{equation*}
u\ne Fu+\mu \gamma\quad\text{for  all } u\in \partial
V_{\rho}\quad\text{and  every } \mu>0.
\end{equation*}
In fact, if not, there exist $u\in \partial V_\rho$ and
$\mu>0$ such that $u=Fu+\mu \gamma$. Then we have$$
u(t)=\gamma(t)\alpha[u]+\int_{a}^{b}
k(t,s)g(s)f(s,u(s),u(\sigma(s)))\,ds+\mu\gamma(t)
$$
and
$$
\alpha[u]=\alpha[\gamma]\alpha[u]+\int_{a}^{b}
\mathcal{K}_{A}(s)g(s)f(s,u(s),u(\sigma(s)))\,ds+\mu\alpha[\gamma],
$$
and therefore
$$\alpha[u]=\frac{1}{1-\alpha[\gamma]}\int_{a}^{b}
\mathcal{K}_{A}(s)g(s)f(s,u(s),u(\sigma(s)))\,ds+\frac{\mu\alpha[\gamma]}{1-\alpha[\gamma]}.$$
Thus we get, for $t\in
[\hat{a},\hat{b}]$,

\begin{align*}
u(t)=&\frac{\gamma(t)}{1-\alpha[\gamma]}\left(\int_{a}^{b}
\mathcal{K}_{A}(s)g(s)f(s,u(s),u(\sigma(s)))\,ds+\mu\a[\gamma] \right)\\ & +\int_{a}^{b}
k(t,s)g(s)f(s,u(s),u(\sigma(s)))\,ds+\mu\c(t) \\  \ge&
\frac{\gamma(t)}{1-\alpha[\gamma]}\int_{\hat{a}}^{\hat{b}}
\mathcal{K}_{A}(s)g(s)f(s,u(s),u(\sigma(s)))\,ds+\int_{\hat{a}}^{\hat{b}}
k(t,s)g(s)f(s,u(s),u(\sigma(s)))\,ds\\ \ge&
\rho f_{\rho ,{\rho /c}}
\;  \left(
\frac{\gamma(t)}{1-\alpha[\gamma]}\int_{\hat{a}}^{\hat{b}}
\mathcal{K}_{A}(s)g(s)\,ds+\int_{\hat{a}}^{\hat{b}}
k(t,s)g(s)\,ds\right).
\end{align*}
Taking the minimum over $[\hat{a},\hat{b}]$ gives
$\rho>\rho$ a contradiction.
\end{proof}
\begin{rem}\label{easyidx0}
We point out, in similar way as in \cite{jwgi-nodea-08}, that a stronger (but easier to check) condition than $(\mathrm{I}_{\protect\rho }^{0})$ is given by the following.
\begin{equation}
  \label{eqmest2}
f_{\rho ,{\rho /c}}\;  \left(\dfrac{c_2\|\gamma\|}{1-\alpha[\gamma]}
\int_{\hat{a}}^{\hat{b}} \mathcal{K}_A(s)g(s)\,ds
+\dfrac{1}{M(\hat{a},\hat{b})}\right)>1,
\end{equation}
where
\begin{equation}\label{Mdef} 
\frac{1}{M(\hat{a},\hat{b})} :=\inf_{t\in [\hat{a},\hat{b}]}\int_{\hat{a}}^{\hat{b}}k(t,s)g(s)\,ds.
\end{equation}
\end{rem}

\begin{rem}\label{whichone} Depending on the nature of the nonlinearity $f$ and due to the way $\theta$ is defined, sometimes it could be useful to take a smaller $[\hat a, \hat b]$ such that $\sigma([\hat a,\hat b])\subseteq [\hat a,\hat b]$. This fact is illustrated in Section 4.
\end{rem}

The above Lemmas can be combined to prove the following Theorem. Here we
deal with the existence of at least one, two or three solutions. We stress
that, by expanding the lists in conditions $(S_{5}),(S_{6})$ below, it is
possible to state results for four or more positive solutions, see for
example the paper by Lan~\cite{kljdeds} for the type of results that might be stated. We omit
the proof which follows directly from the properties of the fixed point index  stated in Lemma \ref{lemind}, (3) and (4).\par
\begin{thm}\label{thmcasesS} The integral equation \eqref{eqhamm} has at least one non-zero solution
in $K$ if any of the following conditions hold.

\begin{enumerate}

\item[$(S_{1})$] There exist $\rho _{1},\rho _{2}\in (0,\infty )$ with $\rho
_{1}/c<\rho _{2}$ such that $(\mathrm{I}_{\rho _{1}}^{0})$ and $(\mathrm{I}_{\rho _{2}}^{1})$ hold.

\item[$(S_{2})$] There exist $\rho _{1},\rho _{2}\in (0,\infty )$ with $\rho
_{1}<\rho _{2}$ such that $(\mathrm{I}_{\rho _{1}}^{1})$ and $(\mathrm{I}%
_{\rho _{2}}^{0})$ hold.
\end{enumerate}
The integral equation \eqref{eqhamm} has at least two non-zero solutions in $K$ if one of
the following conditions hold.

\begin{enumerate}

\item[$(S_{3})$] There exist $\rho _{1},\rho _{2},\rho _{3}\in (0,\infty )$
with $\rho _{1}/c<\rho _{2}<\rho _{3}$ such that $(\mathrm{I}_{\rho
_{1}}^{0}),$ $(
\mathrm{I}_{\rho _{2}}^{1})$ $\text{and}\;\;(\mathrm{I}_{\rho _{3}}^{0})$
hold.

\item[$(S_{4})$] There exist $\rho _{1},\rho _{2},\rho _{3}\in (0,\infty )$
with $\rho _{1}<\rho _{2}$ and $\rho _{2}/c<\rho _{3}$ such that $(\mathrm{I}%
_{\rho _{1}}^{1}),\;\;(\mathrm{I}_{\rho _{2}}^{0})$ $\text{and}\;\;(\mathrm{I%
}_{\rho _{3}}^{1})$ hold.
\end{enumerate}
The integral equation \eqref{eqhamm} has at least three non-zero solutions in $K$ if one
of the following conditions hold.

\begin{enumerate}
\item[$(S_{5})$] There exist $\rho _{1},\rho _{2},\rho _{3},\rho _{4}\in
(0,\infty )$ with $\rho _{1}/c<\rho _{2}<\rho _{3}$ and $\rho _{3}/c<\rho
_{4}$ such that $(\mathrm{I}_{\rho _{1}}^{0}),$ $(\mathrm{I}_{\rho _{2}}^{1}),\;\;(\mathrm{I}%
_{\rho _{3}}^{0})\;\;\text{and}\;\;(\mathrm{I}_{\rho _{4}}^{1})$ hold.

\item[$(S_{6})$] There exist $\rho _{1},\rho _{2},\rho _{3},\rho _{4}\in
(0,\infty )$ with $\rho _{1}<\rho _{2}$ and $\rho _{2}/c<\rho _{3}<\rho _{4}$
such that $(\mathrm{I}_{\rho _{1}}^{1}),\;\;(\mathrm{I}_{\rho
_{2}}^{0}),\;\;(\mathrm{I}_{\rho _{3}}^{1})$ $\text{and}\;\;(\mathrm{I}%
_{\rho _{4}}^{0})$ hold.
\end{enumerate}
\end{thm}
\begin{rem} A similar approach can be used, depending on the signs of $k$ and $\gamma$, to prove the existence of solutions that are negative on sub-interval, non-positive, strictly negative, non-negative and strictly positive. See for example Remark 3.4 of \cite{gijwems} and also Sections 2, 3 and 4 and Remark 4.5 of \cite{ac-gi-at-bvp}.
\end{rem}

\section{An application to a problem with reflection}
We now turn our attention to the first order functional periodic boundary value problem
\begin{equation}\label{eqgenpro2} 
u'(t)=h(t,u(t),u(-t)),\,  t\in I:=[-T,T],
\end{equation}
\begin{equation}\label{prdic2}
u(-T)-u(T)=\alpha[u],
\end{equation}
where $\alpha$ is a linear functional on $C(I)$ given
by
\begin{equation*}
\label{eqalpha2}
\alpha[u]=\int_{-T}^{T} u(s)\,dA(s),
\end{equation*}
involving a Stieltjes integral with a \emph{signed} measure.

We utilize the shift argument of \cite{alb-adr} (a similar idea has been used in \cite{Tor,jwmz-na}), by
fixing $\omega \in \mathbb{R}\setminus \{0\}$ and considering the  equivalent expression
\begin{equation}\label{eqgenpro2b}
u'(t) +\omega u(-t) =h(t,u(t),u(-t))+\omega u(-t)=:f(t,u(t),u(-t)),\,  t\in I,
\end{equation}
with the BCs
\begin{equation}\label{prdicb}
u(-T)-u(T)=\alpha[u].
\end{equation}{}

The Green's function $k$ of the periodic problem
$$u'(t) +\omega u(-t)=f(t,u(t),u(-t)),\ t\in I,\quad u(T)=u(-T)$$
is given by (see \cite{alb-adr,ac-gi-at-bvp})
\begin{equation*}\label{gbarra}
2\sin(\omega T)k(t,s)=\begin{cases} \cos \omega (T-s-t)+\sin \omega (T+s-t) & \text{if}\quad t>|s|,\\\cos \omega (T-s-t)-\sin \omega (T-s+t) & \text{if}\quad |t|<s,\\
\cos \omega (T+s+t)+\sin \omega(T+s-t) & \text{if}\quad |t|<-s,\\
\cos \omega (T+s+t)-\sin \omega(T-s+t) & \text{if}\quad t<-|s|.\end{cases}
\end{equation*}
Note that $k$ only exists when $\omega T\neq l\pi$ for every $l\in\mathbb{Z}$. Hence, \cite[Corollary 3.4]{alb-adr} guarantees that problem  \eqref{eqgenpro2b} -- \eqref{prdicb} is equivalent to the perturbed Hammerstein integral equation
$$u(t)=k(t,-T)\a[u]+ \int_{-T}^Tk(t,s)f(t,u(t),u(-t))dt.$$
Thus, we are working with an equation of the type \eqref{eqHam2} where $$\c(t)=k(t,-T)=\cos\omega t-\sin\omega t=\sqrt 2\sin\(\frac{\pi}{4}-\omega t\).$$
Let $\zeta:=\omega T$. Then we have
$$\|\c\|=\begin{cases} \sqrt 2\sin\(\frac{\pi}{4}+\zeta\) & \text{if}\quad \zeta\in\(0,\frac{\pi}{4}\), \\  \sqrt 2 & \text{if}\quad \zeta\in\[\frac{\pi}{4},\frac{\pi}{2}\).\end{cases}$$
Also, using Lemma 5.5 in \cite{ac-gi-at-bvp}, the constant $c_2$ is given by
$$\|\c\|c_2=\inf_{t\in[\hat a,\hat b]}\c(t)=\begin{cases}\c(\hat b) & \text{if}\quad \zeta\in\(0,\frac{\pi}{4}\]\quad\text{or}\quad \left|\hat a+\frac{\pi}{4\zeta}\right|<\left|\hat b+\frac{\pi}{4\zeta}\right|, \\ \c(\hat a) & \text{if}\quad \zeta\in\(\frac{\pi}{4},\frac{\pi}{2}\]\ \text{and}\  \left|\hat a+\frac{\pi}{4\zeta}\right|\ge\left|\hat b+\frac{\pi}{4\zeta}\right| .\end{cases}$$
The constant $c_1$ was given in \cite{ac-gi-at-bvp} for the case $\hat a+\hat b=1$ and has the following expression
\begin{equation}\label{eqc1}c_1=\frac{(1-\tan\zeta \hat a)(1-\tan\zeta \hat b)}{(1+\tan\zeta \hat a)(1+\tan\zeta \hat b)}.\end{equation}
Observe that in the case $[\hat a, \hat b]=I$, using the fact that $k(t,s)= k(t+1,s+1)$, $ k(t+1,s)= k(t,s+1)$ for $t,s\in[-T,0]$ (cf. \cite{ac-gi-at-bvp}) and formula \eqref{eqc1} for $[\hat a,\hat b]=[0,T]$ we get that
$$c_1=\frac{1-\tan\zeta}{1+\tan\zeta}=\cot\(\frac{\pi}{4}+\zeta\).$$
Consider now the set $\hat S:=\{(\hat a,\hat b)\in\bR^2\ :\ \hat a<\hat b,\ \text{(C2) is satisfied for}\  [\hat a,\hat b]\}$ and $M(\hat a,\hat b)$ defined as in \eqref{Mdef} (with $g\equiv 1$). Since a smaller constant $M(\hat a,\hat b)$ relaxes the growth conditions imposed on the nonlinearity $f$ by the inequality \eqref{eqmest2},
we turn our attention to the quantity
$$\frac{1}{M_{opt}}:=\sup_{(\hat a,\hat b)\in\hat S}\frac{1}{M(\hat a,\hat b)}.$$
A similar study has been done, in the case of second-order BVPs in \cite{gi-pp, jwpomona, jwwcna04} and for fourth order BVPs in \cite{gipp-cant, paola,jwgi-lms-II}.

Before computing this value, we need some relevant information about the kernel $k$.\par
First, observe that with the change of variables $t=\ol x\, T$, $s=\ol y\, T$, $\ol k(x,y)=k(t,s)$,  $a=\ol a\, T$, $b=\ol b\, T$ we have that
$$\frac{1}{M_{opt}}=T\sup_{(\ol a,\ol b)\in\tilde S}\min_{x\in[\ol a,\ol b]}\int_{\ol a}^{\ol b}\ol k(x,y)dy,$$
where $\tilde S:=\{(\ol a,\ol b)\in\bR^2\ :\ (\ol a\, T,\ol a\, T)\in\hat S\}.$\par
There is a symmetry (see \cite{alb-adr}) between the cases $\omega$ and $-\omega$ given by the fact that $\ol k_\omega(x,y)=-\ol k_{-\omega}(-x,-y)$, so we can restrict our problem to the case $\omega>0$.\par
Information on the sign of $\ol k$ is given in the following Lemma which summarizes the findings in \cite{ac-gi-at-bvp,alb-adr}.
\begin{lem}\label{posstrip} Let $\zeta=\omega T$. The following hold:\par
\begin{enumerate}
\item If $\zeta\in(0,\frac{\pi}{4})$, then $\ol k$ is strictly positive in $I^2$.
\item If $\zeta\in(-\frac{\pi}{4},0)$, then $\ol k$ is strictly negative in $I^2$.
\item If $\zeta\in[\frac{\pi}{4},\frac{\pi}{2})$, then $\ol k$ is strictly positive in $$S:=
\[\(-\frac{\pi}{4|\zeta|},\frac{\pi}{4|\zeta|}-1\)\cup\(1-\frac{\pi}{4|\zeta|},\frac{\pi}{4|\zeta|}\)\]\times[-1,1].$$
\item If $\zeta\in(-\frac{\pi}{2},-\frac{\pi}{4}]$, $\ol k$ is strictly negative in $S$.
\end{enumerate}
\end{lem}
First, in \cite{alb-adr}, it was proven that $\ol k$ satisfies the equation $\frac{\partial \ol k}{\partial x}(x,y)+\omega \ol k(-x,y)=0$. Also, the strip $S$ satisfies that, if $(x,y)\in S$, then $(-x,y)\in S$, so, wherever $\ol k\ge0$, $\frac{\partial \ol k}{\partial t}\le0$.  Hence, we have
\begin{equation}\label{momr}\frac{1}{M(\omega)}=T\sup_{(\ol a,\ol b)\in\tilde S}\int_{\ol a}^{\ol b}\ol k(\ol b,y)dy.\end{equation}
 Notice that, fixed $\ol b$, it is of our interest to take $\ol a$ as small as possible (as long as (C2) is satisfied) for we are integrating a positive function on the interval $[\ol a, \ol b]$.\par
With these considerations in mind, we will prove that
$$M_{opt}=\begin{cases} \omega, & \text{if}\quad \zeta\in(0,\frac{\pi}{4}), \\  \frac{\omega}{\cos\zeta}, & \text{if}\quad \zeta\in[\frac{\pi}{4},\frac{\pi}{2}),\end{cases}$$
by studying two cases: (A) and (B).\par
(A) If $\zeta\in(0,\frac{\pi}{4})$, $\ol k$ is positive and
$$\frac{1}{M_{opt}}=T \sup_{\ol b\in[-1,1]}\int_{-1}^{\ol b}\ol k(\ol b,y)dy.$$
(A1) If $\ol b\le0$, let
\begin{align*} g_1(\ol b):= & 2\sin\zeta\int_{-1}^{\ol b}\ol k(\ol b,y)dy \\ = & \int_{-1}^{\ol b}[\cos \zeta (1+y+\ol b)+\sin  \zeta (1+y-\ol b)]ds=\frac{1}{\zeta}\[\sin\zeta(1+2\ol b)-\sin\zeta \ol b+\cos\zeta \ol b-\cos\zeta\].\end{align*}
Then, taking into account that $\ol b\in[-1,0]$ and $\zeta\in\(0,\frac{\pi}{4}\)$ and studying the range of the arguments of the sines and cosines involved, we get
$$g_1'(\ol b)=2 \cos\zeta(1 + 2\ol b)-\sqrt 2 \sin\(\zeta\ol b+\frac{\pi}{4}\)\ge2\frac{\sqrt{2}}{2}-\sqrt{2}\frac{\sqrt{2}}{2}=\sqrt{2}-1>0.$$
Therefore, the maximum of $g_1$ in $[0,1]$ is reached at $0$.\par
(A2) If $\ol b\ge0$,
\begin{align*} g_1(\ol b) & = \int_{-1}^{-\ol b}[\cos \zeta (1+y+\ol b)+\sin  \zeta (1+y-\ol b)]ds+\int_{-\ol b}^{\ol b}[\cos \zeta (1-y-\ol b)+\sin  \zeta (1+y-\ol b)]ds  \\ & =-\frac{1}{\zeta}\[\cos\zeta-\cos\zeta b-2\sin\zeta+\sin \zeta b+\sin\zeta(1-2b)\].\end{align*}
Now, we have
$$g_1'''(\ol b)=-\zeta^2\[8 \cos\zeta(1 - 2\ol b)-\sqrt 2 \sin\(\zeta\ol b+\frac{\pi}{4}\)\]<0.$$
Therefore, $g_1'$ reaches its minimum in $[0,1]$ at $0$ or $1$.
$$g_1'(0)=2 \cos\zeta-1,\ g_1'(1)=\cos\zeta-\sin\zeta>0.$$
Thus, $g_1'>0$ in $[0,1]$, this is, the maximum of $g_1$ in $[0,1]$ is reached at $1$. In conclusion, by the continuity of $g_1$, the maximum of $g_1$ in $[-1,1]$ is reached at $1$ and so
$$\frac{1}{M_{opt}}= T\int_{-1}^{1}\ol k(1,y)dy=T\frac{g_1(1)}{2\sin\zeta}=\frac{T}{\zeta}=\frac{1}{\omega}.$$\par
 Observe now that, since $[\ol a,\ol b]=[-1,1]$, $c=c_1=c_2=\cot\(\frac{\pi}{4}+\zeta\)$.\par
(B) Now assume $\zeta\in[\frac{\pi}{4},\frac{\pi}{2})$. $\ol k$ is positive on $S$.\par
Assume $\ol b>0$. Also, since $\ol k(x,y)=\ol k(-y,-x)$ (see \cite{alb-adr}), fixed $b\in S$, the smallest $\ol a$ that can be taken is $\ol a=1-\frac{\pi}{4\zeta}$, so
$$g_2(\ol b):=2\sin\zeta\int_{1-\frac{\pi}{4\zeta}}^{\ol b} \ol k(\ol b, y)dy=\frac{1}{\zeta}\[\cos\(\frac{\pi}{4}+(\ol b-2)\zeta\)+\cos\(\frac{\pi}{4}+\ol b\zeta\)-\cos \zeta+\sin\((2\ol b-1)\zeta\)\].$$
Thus, we have
$$g_2'''(\ol b)=\zeta^2\[\sin\(\frac{\pi}{4}+(b-2)\zeta\)+\sin\(\frac{\pi}{4}+b\zeta\)-8\cos\(\(1-2b\)\zeta\)\]>\zeta^2\(2-8\frac{\sqrt 2}{2}\)<0.$$
Therefore, $g_2'$ reaches its minimum in $Y:=\[1-\frac{\pi}{4\zeta},\frac{\pi}{4\zeta}\]$ at $1-\frac{\pi}{4\zeta}$ or $\frac{\pi}{4\zeta}$.
$$g_2'\(1-\frac{\pi}{4\zeta}\)=2 \sin\zeta,\ g_2'\(\frac{\pi}{4\zeta}\)=2(\sin\zeta-\cos^2\zeta)>0.$$
Thus, $g_2'>0$ in $Y$, this is, the maximum of $g_2$ in $Y$ is reached at $\frac{\pi}{4\zeta}$
and so
$$T\int_{{1-\frac{\pi}{4\zeta}}}^{\frac{\pi}{4\zeta}}\ol k\(\frac{\pi}{4\zeta},y\)dy=T\frac{g_2\(\frac{\pi}{4\zeta}\)}{2\sin\zeta}=\frac{T\cos\zeta}{\zeta}=\frac{\cos\zeta}{\omega}.$$\par
Now, the case $\ol b\le 0$ can be reduced to the case $\ol b\ge0$ just taking into account that $\ol k(z,y)=\ol k(z+1,y+1)$ for $z,y\in[-1,0]$ (cf. \cite{ac-gi-at-bvp}) and making the change of variables $\ol y= y-1$, so
$$\int_{{1-\frac{\pi}{4\zeta}}}^{\frac{\pi}{4\zeta}}\ol k\(\frac{\pi}{4\zeta},y\)dy=\int_{{-\frac{\pi}{4\zeta}}}^{\frac{\pi}{4\zeta}-1}k\(\frac{\pi}{4\zeta},\ol y+1\)d\ol y=\int_{{-\frac{\pi}{4\zeta}}}^{\frac{\pi}{4\zeta}-1}k\(\frac{\pi}{4\zeta}-1,\ol y\)d\ol y.$$
Hence we have
$$\frac{1}{M_{opt}}= \frac{\cos\zeta}{\omega},$$
 Consider again the case $\zeta\in\(0,\frac{\pi}{4}\)$ and $\hat a_{opt}$,  $\hat b_{opt}$,  $c(\hat a_{opt},\hat b_{opt})$,  the values for which $M_{opt}$ is reached. In the following table we summarize these findings. \par
 \vspace*{1em}
 \begin{center}
\begin{tabular}{|c|c|c|c|c|c|}
\hline  $\zeta$ & $\hat a_{opt}$ &  $\hat b_{opt}$ & $M_{opt}$ & $c(\hat a_{opt},\hat b_{opt})$ & $\|\c\|$ \\ 
\hline $\(0,\frac{\pi}{4}\)$ & $-1$ & $1$  & $\omega$  & $\cot\(\frac{\pi}{4}+\zeta\)$ &  $\sqrt 2\sin\(\frac{\pi}{4}+\zeta\)$\\
\hline 
\end{tabular}
\end{center}
 \vspace*{1em}
 When  $\zeta\in\[\frac{\pi}{4},\frac{\pi}{2}\)$ we have the following.
 \vspace*{1em}
 \begin{center}
\begin{tabular}{|c|c|c|c|c|c|}
\hline  $\zeta$ & $\hat a_{opt}$ &  $\hat b_{opt}$ & $M_{opt}$ &$\|\c\|$ \\ 
\hline $\[\frac{\pi}{4},\frac{\pi}{2}\)$ & $1-\frac{\pi}{4\zeta}$ & $\frac{\pi}{4\zeta}$ &  $\frac{\omega}{\cos\zeta}$ & $\sqrt 2$ \\ 
  & $-\frac{\pi}{4\zeta}$ & $\frac{\pi}{4\zeta}-1$ &  & \\ 
\hline 
\end{tabular}
\end{center}
 \vspace*{1em}
 We point out that in this second case we cannot take an interval $[\hat a,\hat b]$ at which $M_{opt}$ is reached because $c_1$ and $c_2$ tend to zero as we approach that interval, but we may take $[\hat a,\hat b]$ as close as possible to these values, in order to approximate $M_{opt}$. 
 
With all these ingredients we can apply Theorem~\ref{thmcasesS} in order to solve \eqref{eqgenpro2}-\eqref{prdic2} for some given $f$ and $\alpha$.
 
\section{An application to a thermostat problem}
\subsection{The model}
We work here with the model of a light bulb with a temperature regulating system (thermostat). The model includes a bulb in which a metal filament, bended on itself, is inserted with only its two extremes outside of the bulb. There is a sensor that allows to measure the temperature of the filament at a point $\eta$ (see Figure \ref{fig3}). The bulb is sealed with some gas in its interior.
\begin{figure}[hht]
\center{\includegraphics[width=.5\textwidth]{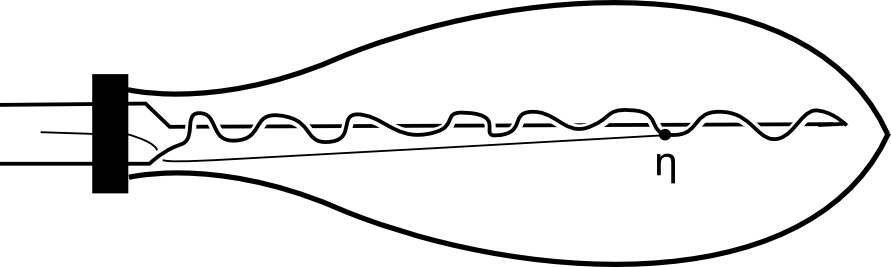}}\caption{Sketch of the light bulb model with a sensor at the point $\eta$.}\label{fig3}
\end{figure}\par
As variables, we take $u$ for the temperature, $t\in [0,1]$ for a point in the filament and $x$ for  the time\footnote{We use this unusual notation in order to be consistent with the rest of the paper. Since we are looking for stationary solutions of the model, the temporal variable will no longer appear after the model is set.}.\par
We control the light bulb via two thermopairs connected to the extremes of the filament. This allows us to measure (and hence modify via a resistance or with some other heating or cooling system) the variation of the temperature with respect to $x$. Also, we will be able to measure the total light ouput of the light bulb.\par
The problem can then be stated as
\begin{equation}\label{eqmodel}\begin{aligned}
\frac{du}{dx}(t,x)= & d_1\frac{d^2u}{dt^2}(t,x)+\int_0^1u^4(y,x)\upsilon(s,t,u(t,x))ds-d_2u^4(t,x)\\ & +j(t,u(t,x))+(d_3+d_4u(t,x))\hat I^2+d_5(u_0-u(t,x)),\end{aligned}
\end{equation}
\begin{equation}\label{bcmodel}
\frac{d u}{dt}(0,x)+d_6\int_0^1u^4(s,x)ds=0,\quad \beta \frac{d u}{dt}(1,x)+u(\eta,x)=0
\end{equation}
where $d_1,\dots,d_5$ and $u_0$ are physical (real) constants that can be determined either theoretically or experimentally; $d_6$, $\hat I$ and $\b$ are real constants to be chosen; $\eta\in[0,1]$ is the position of the sensor at the filament and $\upsilon$ is some real continuous function. We explain now each component of the equation.\par
The term $d_1\frac{d^2u}{dt^2}(t,x)$ comes from the traditional heat equation, $\frac{du}{dx}=d_1\frac{d^2u}{dt^2}$. The integral in the equation stands for the irradiance (that is, power per space unit squared), in form of blackbody radiation, absorbed by the point $t$ and emitted from every other point $s$  of the filament. The function $\upsilon$ gives the rate of this absorption depending on $t$, $s$ and also on $u$, since the reflectivity of metals changes with temperature (see \cite{Uji}). The equation behind the fourth power in the integral comes from the Stefan-Boltzmann equation for blackbody power emission, $j^{\star} = \tilde{k} u^{4}(t,x)$, where $j^{\star}$ is the irradiance and $\tilde{k}$ a constant. Observe that considering the power emission from the rest of the filament is important, since, as early as 1914 (see \cite{Cob}), it has been observed that an interior and much brighter (90 to 100 percent) helix appears in helical filaments of tungsten. Although a $200\,^\circ$C difference would be necessary to account for the extra brightness, experiments show that most of it is due to reflection, being the difference in the temperature less than $5\,^\circ$C.\par
The term $-d_2u^4(t,x)$ accounts again for the Stefan-Boltzmann equation, this time for the irradiance of the point, $j(t,u(t,x))$ for the energy absorbed from the bulb (via reflection and/or blackbody emission) and $(d_3+d_4u(t,x))\hat I^2$ is the heat produced by the intensity of the electrical current, $\hat I$, going through the filament via Ohm's law taking into account a first order approximation of the variation of the resistivity of the metal with temperature. Finally, $d_5(u_0-u(t,x))$ is the heat transfer from the filament to the gas due to Newton's law of cooling, where $u_0$ is the temperature at the interior of the bulb which we may assume constant.\par
The first boundary condition controls the variation of the temperature at the left extreme depending on the total irradiance of the bulb, while the second boundary condition controls the variation of the temperature at the right end of the filament depending on the temperature at $\eta$.\par
Consider now the term
$$\Gamma[u](t,x):=\int_0^1u^4(s,x)\upsilon(s,t,u(t,x))ds.$$
For a fixed $x$, $\Gamma$ is a linear operator on $C[0,1]$. If we consider the wire to be bended on itself, in such a way that every point of the filament touches one and only one other point of the filament, by the continuity of the temperature on the filament, we may take the approximation $\Gamma[u](t,x)=d_7 u^4(\sigma(t,x))$ for some constant $d_7$ and a function $\sigma$ which maps every point in the filament to the other point it is affected by. Clearly, $\sigma$ is an involution.\par
With these ingredients, and looking for stationary solutions of problem \eqref{eqmodel}-\eqref{bcmodel}, we arrive to a BVP of the form
\begin{equation}\label{eq41}
u''(t)+g(t)f(t,u(t),u(\sigma(t)))=0,\ t \in (0,1),
\end{equation}
\begin{equation}\label{eq41bc}
u'(0)+{\alpha}[u]=0,\; \beta u'(1) + u(\eta)=0,\; {\eta}\in [0,1].
\end{equation}
\begin{rem} In some other light bulb model it could happen that every point of the filament is `within reach' of more than one other point, which would mean we could have a multivalued function $\sigma·$ or just two functions $\sigma_1$ and $\sigma_2$ in the equation \eqref{eq41}. Our theory can be extended to the case of having more than one function $\sigma$. A possible approach to the multivalued case would require to extend the theory in \cite{gi-pp}, which is beyond the scope of this paper.
\end{rem}

\subsection{The associated perturbed integral equation}
We now turn our attention to the second order BVP \eqref{eq41}-\eqref{eq41bc}.\par

In a similar way as in \cite{gijwems}, the solution of the BVP~\eqref{eq41}-\eqref{eq41bc} can be expressed as
$$
u(t)=\c(t){\alpha}[u]
+\int_{0}^{1}k(t,s)g(s)f(s,u(s),u(\sigma(s)))ds,
$$
 where $\c(t)=\beta + \eta- t $,  and
\begin{equation*}\label{krnl2}
k(t,s)=\beta +\begin{cases} \eta-s,\ &s\leq \eta\\0,\ &s>\eta
\end{cases}
-\begin{cases} t-s,\ &s\leq t\\ 0,\ &s>t.
\end{cases}
\end{equation*}
Here we focus on the case $\beta\geq 0$ and $0<\beta+\eta<1$, that leads (in similar way to \cite{gijwems}) to the existence of solutions that are positive on a sub-interval. The constant $c$ for this problem (see for example \cite{gi-pp}) is
\begin{equation*}\label{c-bcB}
c=\begin{cases}
\beta / (\beta+\eta), & \text{ for }\hat{b}\leq \eta,\  \beta+\eta\geq \frac{1}{2}, \\
\beta / (1 -(\beta + \eta )), & \text{ for }\hat{b}\leq \eta,\ \beta+\eta< \frac{1}{2}, \\
(\beta + \eta -\hat b) / (\beta+\eta), & \text{ for }\hat{b}> \eta,\  \beta+\eta\geq \frac{1}{2},\\
(\beta + \eta -\hat b)/(1 -(\beta + \eta )), & \text{ for }\hat{b}> \eta,\ \beta+\eta< \frac{1}{2}.
\end{cases}
\end{equation*}
Also, we have
$$
\Phi (s)=\|\gamma\|=\begin{cases}
\beta+\eta, & \text{ for }\beta+\eta\geq \frac{1}{2}, \\
1 -(\beta + \eta ), & \text{ for }\beta+\eta< \frac{1}{2},
\end{cases}
$$
and clearly
$$c_2\|\c\|=\b+\eta-\hat b .$$
Theorem~\ref{thmcasesS} can be applied to this problem for given $f$, $\alpha$ and $g$. We now set  $g\equiv 1$ and recall 
(see \cite{gijwems}) that
$$
\sup_{t \in [0,1]} \int^{1}_{0} |k(t,s)|\,ds =\max \Bigl \{ \beta
+ \frac{1}{2} \eta^{2}, \beta^{2}-\beta +\frac{1}{2}(1-\eta^{2})
\Bigr \}.
$$
Furthermore, note that the solution of the problem
$$w''(t)=-1,\quad w'(0)=0,\quad \b w'(1)+w(\eta)=0,$$
is given by $w(t)=\b+\frac{1}{2}(\eta^2-t^2)$, which implies that
$$w(t)=\int_0^1k(t,s)ds=\b+\frac{1}{2}(\eta^2-t^2).$$
Using this fact and Fubini's Theorem we have that
$$\int_{0}^{1}\mathcal{K}_A(s)ds=\int_{0}^1\int_{{0}}^{{1}} k(t,s)\,dA(t)\,ds=\int_{0}^1\int_{{0}}^{{1}} k(t,s)\,ds\,dA(t)=\a\[\b+\frac{1}{2}(\eta^2-t^2)\].$$
With all these facts, the conditions  \eqref{eqmest} and \eqref{eqmest2} can be rewritten, respectively, for problem \eqref{eq31}--\eqref{eq31bc} as

\begin{equation}\tag{$\tilde{\mathrm{I}}_{\rho }^{1}$} \label{eqmest3} f^{{-\rho},{\rho}} \;<m_\a,
\end{equation}
where

\begin{align*}\frac{1}{m_\a} := & \dfrac{(\b+\eta)\chi_{\[\frac{1}{2},+\infty\)}(\b+\eta)+(1-\b-\eta)\chi_{\(-\infty,\frac{1}{2}\)}(\b+\eta)}{1-\alpha[\beta + \eta- t ]}\cdot
\a\[\b+\frac{1}{2}(\eta^2-t^2)\]\\
& +\max  \left\{ \beta
+ \frac{1}{2} \eta^{2}, \beta^{2}-\beta +\frac{1}{2}(1-\eta^{2})\right\},\end{align*}
$\chi_B$ is the characteristic function of the set $B$;
and
\begin{equation}\tag{$\tilde{\mathrm{I}}_{\rho }^{0}$}
  \label{eqmest4}
f_{\rho ,{\rho /c}}>M_\a,
\end{equation}
where
$$\frac{1}{M_\a}:=\dfrac{\b+\eta-\hat b}{1-\alpha[\beta + \eta- t ]}\cdot
\a\[\int_{\hat{a}}^{\hat{b}}k(t,s)\,ds\]
+\dfrac{1}{M(\hat{a},\hat{b})}.
$$

Therefore, we can restate Theorem \ref{thmcasesS} as follows.
\begin{thm}\label{thmcasesS2}
\label{thmmsol1} Theorem \ref{thmcasesS} is satisfied if we change the conditions $(\mathrm{I}_{\rho }^{0})$ and $(\mathrm{I}_{\rho }^{1})$ by $(\tilde{\mathrm{I}}_{\rho }^{0})$ and $(\tilde{\mathrm{I}}_{\rho }^{1})$ respectively.
\end{thm}
We now illustrate how the behaviour of the deviated argument affects the allowed growth of the nonlinearity $f$.
\begin{ex}\label{ex1}
Take $\eta=1/5$, $\beta=3/5$. It was proven in  \cite{gi-pp} that the optimal interval for such a choice is $[\hat a,\hat b]=[0,3/5]$, for which $M_{opt}=5$, $m=50/31$, $c_1=1/4$.
Consider $\sigma(t)=11t-101t^2+318t^3-394t^4+167t^5$. $\sigma$ satisfies $\sigma([0,1])=[0,1]$ and $\sigma([0,2/5])\subseteq[0,2/5]$ as it is shown in Figure \ref{fig1}.
\begin{figure}[hht]
\center{\includegraphics[width=.5\textwidth]{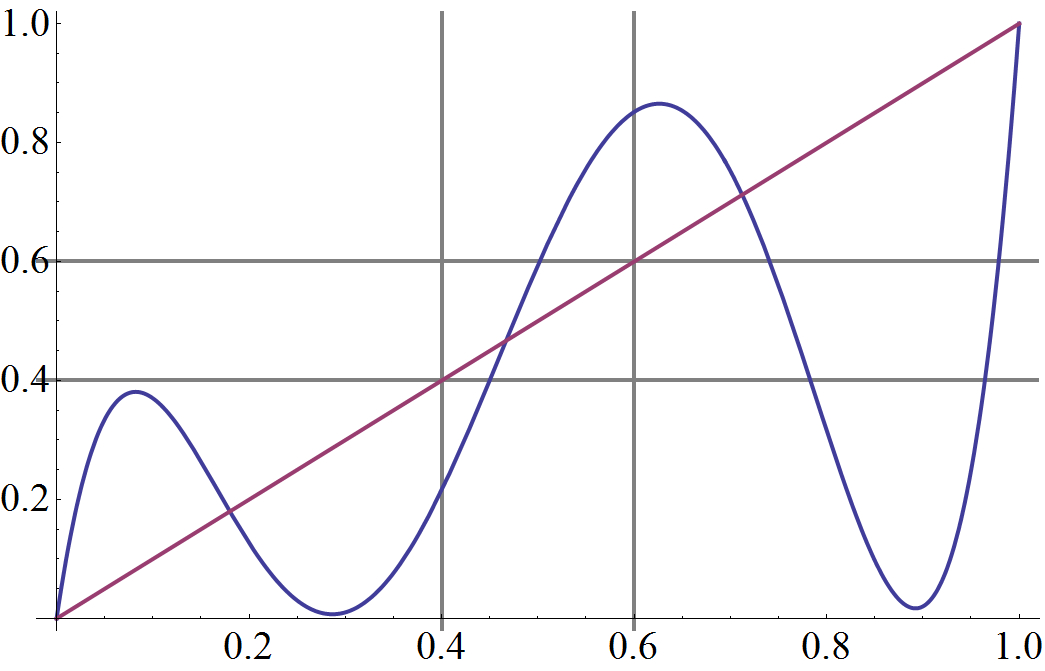}}\caption{Plot of the function $\sigma$ and the identity.}\label{fig1}
\end{figure}\par
Remember that the condition \eqref{eqmest4} is of the form
\begin{equation*}
f_{\rho ,{\rho /c}}(\hat a,\hat b)\; \left(p(\a)q(\hat a,\hat b)+r(\hat a,\hat b)\right)>1
\end{equation*}
where
$$ p(\a)=\dfrac{\|\gamma\|}{1-\alpha[\gamma]},\quad q(\hat a,\hat b)=c_2(\hat a,\hat b)\int_{\hat{a}}^{\hat{b}} \mathcal{K}_A(s)g(s)\,ds\quad\text{and}\quad r(\hat a,\hat b)=\dfrac{1}{M(\hat{a},\hat{b})}.$$
Now, picking up Remark \ref{whichone}, the questions is: Is it worth it to take $[\hat a,\hat b]=[0,3/5]$ or it is preferable to take $[\hat a,\hat b]=[0,2/5]$? Observe that, as mentioned, $\sigma([0,2/5])\subseteq[0,2/5]$ but $\sigma([0,3/5])\not\subseteq[0,3/5]$, which means that the value of $f_{\rho ,{\rho /c}}(\hat a,\hat b)$ can vary considerably from one case to the other. It will be preferable to take  $[\hat a,\hat b]=[0,2/5]$ if and only if
$$\frac{f_{\rho ,{\rho /c}}(0,2/5)}{f_{\rho ,{\rho /c}}(0,3/5)}>\frac{p(\c,\a)q(0,3/5)+r(0,3/5)}{p(\c,\a)q(0,2/5)+r(0,2/5)}.$$
We can compute, a priori, $q(0,3/5)$, $q(0,2/5)$, $r(0,2/5)$ and $r(0,3/5)$, but $f_{\rho ,{\rho /c}}(0,2/5)$ and $f_{\rho ,{\rho /c}}(0,3/5)$ will depend on $f$ and $p(\c,\a)$ on $\a$. As a simple example, if $f$ is zero at a subset of $(2/3,5/3]$ of positive measure, it is clear that the choice to make is $[\hat a,\hat b]=[0,2/5]$.
\end{ex}
\begin{ex}
Continuing with last example, assume now $\a[u]=\l\, u(2/5)$ for some $\l\in(0,5/2)$. $(C_1)$ and $(C_2)$ are satisfied by the properties of the kernel and by the choice of $c_1$. We assume $(C_6)$ is satisfied for the nonlinearity chosen. $(C_4)$ and $(C_7)$ are obviously satisfied. $\mathcal K_a(s)= k((2\l)/5,s)>0$ for every $s\in[0,1]$ by the properties of the kernel, so $(C_3)$ is also satisfied. Last, $0\le\a[4/5- t]=(2\l)/5<1$ and, by the choice of $c_2$, $(C_7)$ is satisfied as well. In this case we have
$m_\a=25/26$, and it is independent of the choice of $[\hat a, \hat b]$. Let us compare the intervals $[0, 2/5]$ and $[0,3/5]$.
$$\frac{1}{M_\a(0,\hat b)}=\frac{4-5\hat b}{1-2\l}\int_{0}^{\hat b}k((2\l)/5,s)\, ds+\inf_{t\in(0,\hat b]}\int_{0}^{\hat b} k(t,s)\, ds.$$
It was proven in \cite{gi-pp} that, for $0\le \hat{a}<\hat{b}<\b+\eta$,
$$\inf_{t\in(0,\hat b]}\int_{0}^{\hat b} k(t,s)\, ds=\int_{0}^{\hat b} k(\hat b,s)\, ds.$$
Hence,
$$M_\a(0,2/5)=\begin{cases} \frac{50(1-2\l)}{43+2\l} & \text{if}\quad \l\in[1,5/2),\\
 \frac{50(1-2\l)}{(7-2\l)(5+4\l)} & \text{if}\quad \l\in(0,1),\end{cases}
$$
$$M_\a(0,3/5)=\begin{cases} \frac{25+50\l}{19+4\l} & \text{if}\quad \l\in[1,5/2),\\
 \frac{50(1+2\l)}{29+20\l-4\l^2} & \text{if}\quad \l\in(0,1).\end{cases}
$$
 Figure \ref{fig2} shows how these two values vary depending on $\l$.
\begin{figure}[hht]
\center{\includegraphics[width=.5\textwidth]{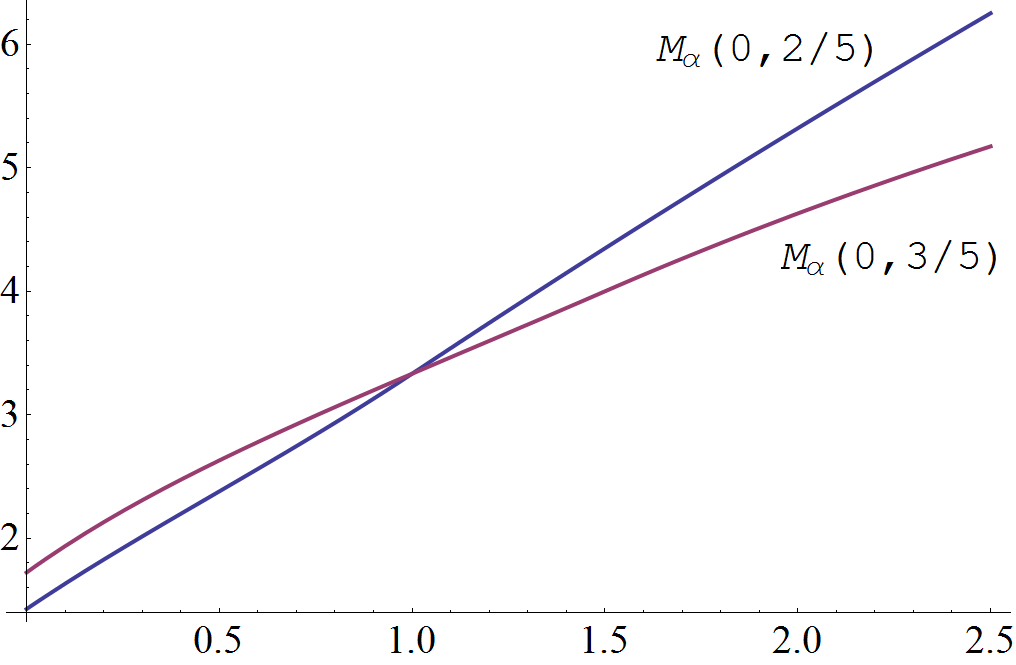}}\caption{Plot of $M_\a(0,2/5)$ and $M_\a(0,3/5)$ depending on $\l$.}\label{fig2}
\end{figure}\par
If we take an specific value for $\l$, say $\l=1$, we get
$M_\a(0,2/5)=M_\a(0,3/5)=10/3$, and so it is more convenient to take $[\hat a,\hat b]=[0,2/5]$. The reason for this is that $f_{\rho,\rho/c}(0,2/5)\ge f_{\rho,\rho/c}(0,3/5)$ independently of $f$, and so $\mathrm{I}_{\rho }^{0}$ is more easily satisfied.\par
Observe in Figure \ref{fig2} that the graphs of $M_\a(0,2/5)(\l)$ and $M_\a(0,3/5)(\l)$ cross at $\l=1$. If $f$ is continuous and $f_{\rho,\rho/c}(0,2/5)> f_{\rho,\rho/c}(0,3/5)$, since $M_\a(0,2/5)(1)$ is a better choice than $M_\a(0,3/5)(1)$, by the continuity of $f$, so it will be in a neighborhood of $1$. 
That shows that the condition $M_\a(0,2/5)(\l)<M_\a(0,3/5)(\l)$ may help but is not deciding when choosing the interval.
\end{ex}

\begin{acknowledgements}
The authors want to express their gratitude towards the physicist Santiago Codesido, whose ideas helped develop the model in Section 4.1.
This paper was partially written during a visit of G. Infante to the 
Departamento de An\'alise Matem\'atica of the Universidade de Santiago 
de Compostela. G. Infante is grateful to the people of the 
aforementioned Departamento for their kind and warm hospitality.
\end{acknowledgements}

\end{document}